\documentclass[reqno]{amsart}
\usepackage{amsmath}
\usepackage{amssymb}
\usepackage{amsthm}
\begin{document}
\def\eq#1{{\rm(\ref{#1})}}
\theoremstyle{plain}
\newtheorem*{theo}{Theorem}
\newtheorem*{ack}{Acknowledgements}
\newtheorem{thm}{Theorem}[section]
\newtheorem{lem}[thm]{Lemma}
\newtheorem{prop}[thm]{Proposition}
\newtheorem{cor}[thm]{Corollary}
\theoremstyle{definition}
\newtheorem{dfn}[thm]{Definition}
\newtheorem*{rem}{Remark}
\def\coker{\mathop{\rm coker}}
\def\ind{\mathop{\rm ind}}
\def\Re{\mathop{\rm Re}}
\def\vol{\mathop{\rm vol}}
\def\Im{\mathop{\rm Im}}
\def\im{\mathop{\rm im}}
\def\Hol{{\textstyle\mathop{\rm Hol}}}
\def\O{{\mathbin{\mathbb O}}}
\def\C{{\mathbin{\mathbb C}}}
\def\R{{\mathbin{\mathbb R}}}
\def\N{{\mathbin{\mathbb N}}}
\def\Z{{\mathbin{\mathbb Z}}}
\def\L{{\mathbin{\mathcal L}}}
\def\X{{\mathbin{\mathcal X}}}
\def\al{\alpha}
\def\be{\beta}
\def\ga{\gamma}
\def\de{\delta}
\def\ep{\epsilon}
\def\io{\iota}
\def\ka{\kappa}
\def\la{\lambda}
\def\ze{\zeta}
\def\th{\theta}
\def\vt{\vartheta}
\def\vp{\varphi}
\def\si{\sigma}
\def\up{\upsilon}
\def\om{\omega}
\def\De{\Delta}
\def\Ga{\Gamma}
\def\Th{\Theta}
\def\La{\Lambda}
\def\Om{\Omega}
\def\Up{\Upsilon}
\def\sm{\setminus}
\def\na{\nabla}
\def\pd{\partial}
\def\op{\oplus}
\def\ot{\otimes}
\def\bigop{\bigoplus}
\def\iy{\infty}
\def\ra{\rightarrow}
\def\longra{\longrightarrow}
\def\dashra{\dashrightarrow}
\def\t{\times}
\def\w{\wedge}
\def\bigw{\bigwedge}
\def\d{{\rm d}}
\def\bs{\boldsymbol}
\def\ci{\circ}
\def\ti{\tilde}
\def\ov{\overline}
\title[Closed $G_2$-Structures on $T^*X\t\R$]{A Note on Closed $G_2$-Structures and $3$-Manifolds}

\author[Cho, Salur, and Todd]{Hyunjoo Cho, Sema Salur, and A. J. Todd}

\address {Department of Mathematics, University of Rochester, Rochester, NY, 14627}
\email{cho@math.rochester.edu}

\address {Department of Mathematics, University of Rochester, Rochester, NY, 14627}
\email{salur@math.rochester.edu}

\address {Department of Mathematics, University of California - Riverside, Riverside, CA, 92521}
\email{ajtodd@math.ucr.edu}

\begin{abstract}
This article shows that given any orientable $3$-manifold $X$, the $7$-manifold $T^*X\t\R$ admits a closed $G_2$-structure $\vp=\Re\Om+\om\w\d t$ where $\Om$ is a certain complex-valued $3$-form on $T^*X$; next, given any $2$-dimensional submanifold $S$ of $X$, the conormal bundle $N^*S$ of $S$ is a $3$-dimensional submanifold of $T^*X\t\R$ such that $\vp|_{N^*S}\equiv 0$. A corollary of the proof of this result is that $N^*S\t\R$ is a $4$-dimensional submanifold of $T^*X\t\R$ such that $\vp|_{N^*S\t\R}\equiv0$.
\end{abstract}

\date{}
\maketitle
\section*{Introduction}
Berger's classification of the possible holonomy groups for a given Riemannian manifold includes the exceptional Lie group $G_2$ as the holonomy group of a $7$-dimensional manifold. On a given a $7$-dimensional manifold with holonomy group a subgroup of $G_2$, there is a nondegenerate differential $3$-form $\vp$ which is torsion-free, that is, $\na\vp=0$. This torsion-free condition is equivalent to $\vp$ being closed and coclosed. Much work has been done to study manifolds with $G_2$-holonomy, e. g. \cite{Br1} and \cite{Jo}, but the condition $\vp$ be coclosed is a very rigid condition. If we drop the coclosed condition, then we are studying manifolds with a closed $G_2$-structure. In particular, manifolds with closed $G_2$-structures have been studied in the articles \cite{Br2}, \cite{BrXu} and \cite{ClIv}; these papers focused predominantly on the metric defined by the nondegenerate closed $3$-form $\vp$. We shift our focus to the form $\vp$ itself, in particular, results which depend on $\vp$ being nondegenerate and closed.

Links between Calabi-Yau geometry and $G_2$ geometry have been explored in the context of mirror symmetry by Akbulut and Salur \cite{AkSa2}. Of course, the connections between symplectic geometry and Calabi-Yau geometry are obvious; moreover, connections between symplectic and contact geometry have been explored for centuries. Thus, it seems completely natural to try to find connections between contact geometry and $G_2$ geometry. The study of these relationships is an ongoing project that begins with the work by Arikan, Cho and Salur \cite{ACS}, and the purpose of the current article (and many upcoming articles) is to continue this study by examining the geometry of closed $G_2$-structures as an analogue of symplectic geometry.

Treating symplectic geometry and $G_2$ geometry as being analogues of one another is not new. In \cite{BrGr} and \cite{Gray} vector cross products (on manifolds) are studied in a general setting; in particular, symplectic geometry is the geometry of $1$-fold vector cross products, better known as almost complex structures, and $G_2$ geometry is the geometry of $2$-fold vector cross products in dimension $7$. Further, in all cases, one can show that using the metric, there is a nondegenerate differential form of degree $k+1$ associated to a $k$-fold vector cross product. This yields the symplectic form associated to almost complex structures and the $G_2$ $3$-form $\vp$ associated to $2$-fold vector cross products in dimension $7$. Examples of manifolds with $G_2$ structures satisfying various conditions (including closed $G_2$ structures) are studied and classified in \cite{CMS}, \cite{Fe1}, \cite{Fe2}, \cite{FeGr} and \cite{FeIg}.

This article is based on two elementary results from symplectic geometry: $1)$ the cotangent bundle $T^*X$ of any $n$-dimensional manifold $X$ admits a symplectic form, and $2)$ the conormal bundle $N^*S$ of any $k$-dimensional submanifold $S$ of $X$ with $k<n$ is a Lagrangian submanifold of $T^*X$. It is a well-known result (\cite{Br2}) that any $7$-manifold which is spin admits a $G_2$-structure. We show that for an orientable $3$-manifold $X$, $T^*X\t\R$ is spin, and hence, admits a $G_2$-structure. Next, we calculate an explicit formula in coordinates for a $G_2$-structure $\vp$ on $T^*X\t\R$. Using this, we prove:

\begin{theo}
Let $S$ be a $2$-dimensional submanifold of $X$, and let $N^*S$ denote the conormal bundle of $S$ in $T^*X$. Let $i:N^*S\hookrightarrow(T^*X\t\R,\vp)$ be the inclusion map. Then $i^*\vp=0$.
\end{theo}

\section{$G_2$ Geometry}
We begin with a description of $G_2$ geometry in flat space, then we consider this geometry on $7$-manifolds.

Consider the octonions $\O$ as an $8$-dimensional real vector space. This becomes a normed algebra when equipped with the standard Euclidean inner product on $\R^8$. Further, there is a cross product operation given by $u\t v=\Im(\overline{v}u)$ where $\overline{v}$ is the conjugate of $v$ for $u,v\in\O$. This is an alternating form on $\Im\O$ since, for any $u\in\Im\O$, $u^2\in\Re\O$. We now define a $3$-form on $\Im\O$ by $\vp(u,v,w)=\langle u\t v,w\rangle$. In terms of the standard orthonormal basis $\{e_1,\ldots,e_7\}$ of $\Im\O$, $\vp_0=e^{123}+e^{145}+e^{167}+e_{246}-e^{257}-e^{347}-e^{356}$ where $e^{ijk}=e^i\w e^j\w e^k$. Under the isomorphism $\R^7\simeq\Im\O$, with coordinates on $\R^7$ given by $(x^1,\ldots,x^7)$, we have $\vp_0=\d x^{123}+\d x^{145}+\d x^{167}+\d x^{246}-\d x^{257}-\d x^{347}-\d x^{356}$.

\begin{dfn}
Let $M$ be a $7$-dimensional manifold. $M$ has a \emph{$G_2$-structure} if there is a smooth $3$-form $\vp\in\Om^3(M)$ such that at each $x \in M$, the pair $(T_{x}(M),\varphi(x))$ is isomorphic to $(T_{0}(\mathbb{R}^{7}),\varphi_{0})$. A $7$-manifold $M$ has a \emph{closed} $G_2$-structure if the $3$-form $\vp$ is also closed, $\d\vp=0$.
\end{dfn}

Equivalently, a smooth $7$-dimensional manifold $M$ has a $G_2$-structure if its tangent frame bundle reduces to a $G_2$-bundle. For a manifold with $G_2$-structure $\vp$, there is a natural Riemannian metric and orientation induced by $\vp$ given by $(Y\lrcorner\vp)\w(\tilde{Y}\lrcorner\vp)\w\vp=\langle Y,\tilde{Y}\rangle_{\vp}dvol_M.$ In particular, the $3$-form $\vp$ is nondegenerate.

\begin{rem}
One defines a $G_2$-manifold as a smooth $7$-manifold with torsion-free $G_2$-structure, i. e., $\na\vp=0$ where $\na$ is the Levi-Civita connection of the metric $\langle\cdot,\cdot\rangle_{\vp}$. This means that $(M,\vp)$ has holonomy group contained in $G_2$ and that $d\varphi=\d{\ast}\varphi=0$. For our purposes, we do not assume that $\d*\vp=0$, so that $(M,\vp)$ will be a manifold with closed $G_2$-structure as defined above.
\end{rem}

\section{The Cotangent Bundle}
This material can be found in most introductions to symplectic geometry or topology, e. g. \cite{daSi} or \cite{McSa}.

Let $X$ be any $n$-dimensional manifold (all manifolds under consideration are assumed to be $C^{\infty}$); let $(U;x_1,\ldots,x_n)$ be a coordinate chart for $X$, so that for each $1\leq i\leq n$, $x_i:U\to\R$. For any point $x\in X$, the differentials $(\d x_1)_x,\ldots,(\d x_n)_x$ form a basis for the cotangent space $T^*_xX$ at $x$; hence, for any covector $\xi\in T^*_xX$, $\xi=\sum_i\xi_i(\d x_i)_x$ with the $\xi_i\in\R$. This yields a coordinate chart $(T^*U;x_1,\ldots,x_n,\xi_1,\ldots,\xi_n)$ for the cotangent bundle of $X$ associated to the coordinates $x_1,\ldots,x_n$ on $X$.

Using these coordinates, the so-called tautological $1$-form on $T^*U$ is defined by $\al:=\sum_i\xi_i\d x_i$. This definition is invariant under changes of coordinates: Let $(T^*V;y_1,\ldots,y_n,\eta_1,\ldots,\eta_n)$ be an overlapping coordinate chart; then $$\eta_j=\sum_i\xi_i\frac{\pd x_i}{\pd y_j}$$ and $$\d x_i=\sum_j\frac{\pd x_i}{\pd y_j}\d y_j.$$ Therefore, in the overlap $$\sum_i\xi_i\d x_i=\sum_i\xi_i(\sum_j\frac{\pd x_i}{\pd y_j}\d y_j)=\sum_{i,j}\xi_i\frac{\pd x_i}{\pd y_j}\d y_j=\sum_j\eta_j\d y_j$$ Now, define a $2$-form by $\om:=-\d\al=\sum_i\d x_i\w\d\xi_i$; $\om$ is also independent of the choice of coordinates, so $\om$ is a symplectic form on the cotangent bundle $T^*X$, called the canonical symplectic form.

Now, let $S$ be any $k$-dimensional submanifold of $X$ with $k<n$. Recall that the conormal space $N^*_xS$ at $x\in S$ is given by $$N^*_xS:=\{\xi\in T^*X:\xi(v)=0\text{ for all }v\in T_xS\},$$ and the conormal bundle $N^*S$ of $S$ by $$N^*S=\{(x,\xi)\in T^*X:x\in S, \xi\in N^*_xS\}.$$ Let $(U;x_1,\ldots,x_n)$ be a coordinate chart on $X$ such that $S\cap U$ is given by the equations $$x_{k+1}=\cdots=x_n=0.$$ In the associated cotangent bundle coordinate chart $(T^*U;x_1,\ldots,x_n,\xi_1,\ldots,\xi_n)$, any $\xi\in N^*S\cap T^*U$ is given by $$\xi=\sum_{i=1}^k\xi_i\d x_i$$ since $x_{k+1}=\cdots=x_n=0$ on $S$; further, since $\xi\in N^*S$ and $T^*_xS$ is spanned by $$(\frac{\pd}{\pd x_1})_x,\ldots,(\frac{\pd}{\pd x_k})_x,$$ we find that $$0=\xi((\frac{\pd}{\pd x_i})_x)=\xi_i\text{, for all }1\leq i\leq k$$ Hence, $N^*S\cap T^*U$ is described by the equations $x_{k+1}=\cdots=x_n=\xi_1=\cdots=\xi_k=0$, so $N^*S$ is an $n$-dimensional submanifold of $T^*X$; further, $\al=\sum_{i=1}^n\xi_i\d x_i$ when restricted to $N^*S$ is zero, so $\om|_{N^*S}\equiv 0$. Thus, $N^*S$ is a Lagrangian submanifold of $T^*X$.

\section{$G_2$-Structures on $T^*X\t\R$}

References for this section are  \cite{AkSa1}, \cite{Br2}, \cite{LaMi} and \cite{Mi}.

Recall that for each $n\geq 3$, the Lie group $SO(n)$ is connected, and it has a double-covering map $\iota : Spin(n) \rightarrow SO(n)$ where the Lie group $Spin(n)$ is a compact, connected, simply-connected Lie group. An oriented Riemannian manifold $(X,g)$ has $SO(n)$ as structure group on the tangent bundle. An spin structure on $(X,g)$ is a $Spin(n)$-principal bundle over $X$, together with a bundle map $ \pi : P_{Spin(n)}X \rightarrow P_{SO(n)}X$ such that $\pi(pg)=\pi(p)\iota(g)$ for $p \in P_{Spin(n)}X, g\in Spin(n)$. A spin manifold is an oriented Riemannian manifold with a spin structure on its tangent bundle.

\begin{thm}
For any oriented Riemannian $3$-manifold $X$, $T^*X\t\R$ has a $G_2$-structure.
\end{thm}

\begin{proof}
Let $X$ be an oriented Riemannian $3$-manifold. We know that every orientable $3$-manifold is parallelizable, and since a framing on a bundle gives a spin structure, $X$ is a spin manifold.

Now we check that $T^{\ast}X$ is itself a spin manifold for a spin manifold $X$. Recall $T^{\ast}X$ carries a canonical $1$-form $\sum p_{i}dx^{i}$, where $p_{i}$ are coordinates in $T^{\ast}X$ and $x^{i}$ are coordinates on $X$. One can define the map $T^{\ast}X\times Spin(3) \rightarrow T^{\ast}X$ by $(p_{i},g)\mapsto p_{i}g|_{x^{i}}=p_i(v_{i}g)$ for $v_{i} \in T_{x^{i}}X$; it is well-defined because of the spin structure on $X$. This implies that $T^{\ast}X$ has an induced spin structure from $X$. Another way to see this, one can find a bundle map of the principal bundles $b :P_{Spin}T^{\ast}X \rightarrow P_{Spin}TX$ over $X$ because the map $\mathfrak{F} : \mathcal{E}(T^{\ast}X)\rightarrow \mathcal{E}(TX)$ is a linear, where $\mathcal{E}$ is the set of sections of the bundle. This gives a map $P_{Spin(3)}X \rightarrow P_{SO(3)}X$ for which the following diagram commutes:
$$\begin{array}{ccc}
      P_{Spin(3)}T^{\ast}X & \longrightarrow^{\pi_{1}} & P_{SO(3)}T^{\ast}X \\
      &  & \\
     \downarrow &  & \downarrow  \\
      &  & \\
      P_{Spin(3)}TX & \longrightarrow_{\pi_{2}}  & P_{SO(3)}TX \\
  \end{array}
$$

Let $E(\xi)$ be the total space of $\xi$ with base space $X$. A principal $Spin(3)$-bundle $E(\xi)\times Spin(3) \rightarrow X$ induces a principal $Spin(4)$-bundle $E(\xi) \times Spin(4)\rightarrow X$ which is itself induced from $Spin(3)(\simeq SU(2)\simeq S^{3}) \rightarrow Spin(4)(\simeq SU(2) \times SU(2))$; hence, a spin structure on $\xi$ gives a spin structure on $\xi \oplus\varepsilon$, so $T^{\ast}X \times \mathbb{R}$ admits a spin structure. By \cite[p. 4]{AkSa1}, we conclude $T^{\ast}X \times \mathbb{R}$ is a smooth $7$-dimensional manifold with a $G_{2}$-structure.
\end{proof}

Let $X$ be an orientable $3$-dimensional manifold with symplectic cotangent bundle $(M=T^*X,\om:=-\d\al)$ where $\al$ is the tautological $1$-form on $M$. If $x_1,x_2,x_3,\xi_1,\xi_2,\xi_3$ are the standard cotangent bundle coordinates associated to the coordinates $x_1,x_2,x_3$ on $X$, define a complex-valued $(3,0)$-form on $M$ by $$\Om=(\d x_1+i\d\xi_1)\w(\d x_2+i\d\xi_2)\w(\d x_3+i\d\xi_3).$$ Consider $M\t\R$. This is a $7$-manifold with coordinates $x_1,x_2,x_3,\xi_1,\xi_2,\xi_3,t$ where $t$ is $\R$ coordinate. Finally, define $\vp=\Re(\Om)+\om\w\d t$.

We show that this defines a $G_2$ structure on $M\t\R$, that is, we exhibit an isomorphism of $(T_{(p,t)}(M\t\R),\vp_{(p,t)})$ with $(\R^7,\vp_0)$. We first calculate $$\Om=(\d x_1+i\d\xi_1)\w((\d x_2\w\d x_3-\d\xi_2\w\d\xi_3)+i(\d x_2\w\d\xi_3-\d x_3\w\d\xi_2))$$ $$=(\d x_1\w\d x_2\w\d x_3-\d x_1\w\d\xi_2\w\d\xi_3+\d x_2\w\d\xi_1\w\d\xi_3-\d x_3\w\d\xi_1\w\d\xi_2)$$ $$+i(\d x_1\w\d x_2\w\d\xi_3-\d x_1\w\d x_3\w\d\xi_2+\d x_2\w\d x_3\w\d\xi_1-\d\xi_1\w\d\xi_2\w\d\xi_3),$$ so we find that $$\vp=\Re\Om+\om\w\d t$$ $$=\d x_1\w\d x_2\w\d x_3-\d x_1\w\d\xi_2\w\d\xi_3+\d x_2\w\d\xi_1\w\d\xi_3-\d x_3\w\d\xi_1\w\d\xi_2$$ $$+\d x_1\w\d\xi_1\w\d t+\d x_2\w\d\xi_2\w\d t+\d x_3\w\d\xi_3\w\d t.$$

Now, for $p=(x,\xi)\in T^*U$, let $$\{(\frac{\pd}{\pd x_1})_p,(\frac{\pd}{\pd x_2})_p,(\frac{\pd}{\pd x_3})_p,(\frac{\pd}{\pd \xi_1})_p,(\frac{\pd}{\pd \xi_2})_p,(\frac{\pd}{\pd \xi_3})_p\}$$ be the basis for the tangent space $T_pM$ at $p$ with respect to the cotangent coordinates on $M$; let $(\frac{\pd}{\pd t})_q$ be the basis for $T_q\R$ with respect to the coordinate $t$ on $\R$. Let $(x_1,\ldots,x_7)$ be the standard Euclidean coordinates on $\R^7$. Define an isomorphism of the tangent vector spaces by $\Phi:T_0\R^7\to T_{(p,q)}(M\t\R)$ by $$\Phi(\frac{\pd}{\pd x_1})_0=-(\frac{\pd}{\pd x_3})_p\text{, }\Phi(\frac{\pd}{\pd x_2})_0=(\frac{\pd}{\pd x_2})_p\text{, }\Phi(\frac{\pd}{\pd x_3})_0=(\frac{\pd}{\pd x_1})_p$$ $$\Phi(\frac{\pd}{\pd x_4})_0=(\frac{\pd}{\pd \xi_1})_p\text{, }\Phi(\frac{\pd}{\pd x_5})_0=(\frac{\pd}{\pd \xi_2})_p\text{, }\Phi(\frac{\pd}{\pd x_6})_0=(\frac{\pd}{\pd \xi_3})_p$$ $$\Phi(\frac{\pd}{\pd x_7})_0=-(\frac{\pd}{\pd t})_q.$$ This induces an isomorphism of the cotangent vector spaces $\Phi^*:T^*_{(p,q)}(M\t\R)\to T^*_0\R^7$ where $$\Phi^*(\d x_1)_p=(\d x_3)_0\text{, }\Phi^*(\d x_2)_p=(\d x_2)_0\text{, }\Phi^*(\d x_3)_p=-(\d x_1)_0$$ $$\Phi^*(\d \xi_1)_p=(\d x_4)_0\text{, }\Phi^*(\d \xi_2)_p=(\d x_5)_0\text{, }\Phi^*(\d \xi_3)_p=(\d x_6)_0$$ $$\Phi^*(\d t)_t=-(\d x_7)_0.$$ Then $$\Phi^*\vp=\Phi^*\d x_1\w\Phi^*\d x_2\w\Phi^*\d x_3-\Phi^*\d x_1\w\Phi^*\d\xi_2\w\Phi^*\d\xi_3+\Phi^*\d x_2\w\Phi^*\d\xi_1\w\Phi^*\d\xi_3-\Phi^*\d x_3\w\Phi^*\d\xi_1\w\Phi^*\d\xi_2$$ $$+\Phi^*\d x_1\w\Phi^*\d\xi_1\w\Phi^*\d t+\Phi^*\d x_2\w\Phi^*\d\xi_2\w\Phi^*\d t+\Phi^*\d x_3\w\Phi^*\d\xi_3\w\Phi^*\d t$$ $$=\d x_3\w\d x_2\w(-\d x_1)-\d x_3\w\d x_5\w\d x_6+\d x_2\w\d x_4\w\d x_6-(-\d x_1)\w\d x_4\w\d x_5$$ $$+\d x_3\w\d x_4\w-(\d x_7)+\d x_2\w\d x_5\w(-\d x_7)+(-\d x_1)\w\d x_6\w(-\d x_7)$$ $$=\d x_1\w\d x_2\w\d x_3+\d x_1\w\d x_4\w\d x_5+\d x_1\w\d x_6\w\d x_7+\d x_2\w\d x_4\w\d x_6$$ $$-\d x_2\w\d x_5\w\d x_7-\d x_3\w\d x_4\w\d x_7-\d x_3\w\d x_5\w\d x_6=\vp_0.$$

In order to show that $\vp$ is independent of the choice of coordinates, it is enough prove that our coordinate definition of $\Om$ on $T^*X$ is independent of the choice of coordinates on $T^*X$. This calculation follows as in \cite{HaLa}. Let $L$ be an oriented $3$-dimensional subspace of $T^*X$. Let $\{f_1,f_2,f_3\}$ be any oriented linearly independent subset of $L$, let $\{e_1,e_2,e_3,Je_1,Je_2,Je_3\}$ denote the standard basis for $\R^6$ and let $A$ be the map defined by $e_j\mapsto f_j$ and $Je_j\mapsto Jf_j$. Now, $A\in GL(3;\C)$, that is, $A$ is complex linear and $A(e_1\w e_2\w e_3)=f_1\w f_2\w f_3$. Let $\{\tilde{f}_1,\tilde{f}_2,\tilde{f}_3\}$ be an oriented orthonormal basis for $L$, and let $B$ be the map defined by $\tilde{f}_j\mapsto f_j$. Then $f_1\w f_2\w f_3=(det B)\tilde{f}_1\w\tilde{f}_2\w\tilde{f}_3$. Since $\Om(A(e_1\w e_2\w e_3))=det_{\C}A$ and $\Om=\Re\Om+i\Im\Om$, we have $$(det B)[\Re\Om(\tilde{f}_1\w\tilde{f}_2\w\tilde{f}_3)]=\Re(det_{\C}A)$$ and $$(det B)[\Im\Om(\tilde{f}_1\w\tilde{f}_2\w\tilde{f}_3)]=\Im(det_{\C}A).$$ Hence, $$(\Re\Om(f_1\w f_2\w f_3))^2+(\Im\Om(f_1\w f_2\w f_3))^2$$ $$=|det_{\C}A|^2=det_{\R}A=vol(A(e_1\w Je_1\w e_2\w Je_2\w e_3\w Je_3))$$ $$=vol(A(e_1\w e_2\w e_3\w Je_1\w Je_2\w Je_3))=vol(f_1\w f_2\w f_3\w Jf_1\w Jf_2\w Jf_3)$$ $$=(det B)^2vol(\tilde{f}_1\w\tilde{f}_2\w\tilde{f}_3\w J\tilde{f}_1\w J\tilde{f}_2\w J\tilde{f}_3).$$

\section{The Conormal Bundle}
Let $X$ be a $3$-dimensional manifold as in the previous section. Then $T^*X\t\R$ admits the $G_2$-structure $\vp=\Re\Om+\om\w\d t$. Let $S$ be a $2$-dimensional submanifold of $X$, and let $(U,x_1,x_2,x_3)$ be a coordinate chart on $X$ such that $S\cap U$ is given by the equation $x_3=0$. For the associated cotangent coordinate chart $(T^*U,x_1,x_2,x_3,\xi_1,\xi_2,\xi_3)$, any $\xi\in N^*S\cap T^*U$ is given by $$\xi=\xi_1\d x_1+\xi_2\d x_2$$ since $x_3=0$ on $S$; further, since $\xi\in N^*S$ and $T^*_xS$ is spanned by $$(\frac{\pd}{\pd x_1})_x,(\frac{\pd}{\pd x_2})_x,$$ we find that $$0=\xi((\frac{\pd}{\pd x_1})_x)=\xi_1$$ and $$0=\xi((\frac{\pd}{\pd x_2})_x)=\xi_2.$$ Hence, $N^*S\cap T^*U$ is described by the equations $x_3=\xi_1=\xi_2=0$. Thus, $N^*S$ is a $3$-dimensional submanifold of $T^*X$.

\begin{prop}
Let $i:N^*S\hookrightarrow(T^*X\t\R,\vp)$ be the inclusion. Then $i^*\vp=0$.
\end{prop}

\begin{proof}
In this case, $N^*S\cap(T^*U\t\R)$ is given by the equations $x_3=\xi_1=\xi_2=t=0$ where $t$ is the $\R$ coordinate. Recall that in this coordinate system, we have $$\vp=\d x_1\w\d x_2\w\d x_3-\d x_1\w\d\xi_2\w\d\xi_3+\d x_2\w\d\xi_1\w\d\xi_3-\d x_3\w\d\xi_1\w\d\xi_2$$ $$+\d x_1\w\d\xi_1\w\d t+\d x_2\w\d\xi_2\w\d t+\d x_3\w\d\xi_3\w\d t.$$ Hence, for $p\in N^*S$, we have $(i^*\vp)_p=\vp_p|_{T_p(N^*S)}=0$ since every term contains a factor which on $N^*S$ is zero.
\end{proof}

\begin{cor}
Let $i:N^*S\t\R\hookrightarrow(T^*X\t\R,\vp)$ be the inclusion. Then $i^*\vp=0$.
\end{cor}

\begin{proof}
Note that $(N^*S\t\R)$ is a $4$-dimensional submanifold of $T^*X$, given by the equations $x_3=\xi_1=\xi_2=0$ on $(N^*S\t\R)\cap(T^*U\t\R)$.
\end{proof}

Note that in contrast to the symplectic case, $S$ must be $2$-dimensional; if $S$ is $1$-dimensional, then $\vp|_{N^*S}\neq 0$.


\begin{thebibliography}{[FP]}

\bibitem{AkSa1} Akbulut, S. and Salur, S., {\it Calibrated Manifolds and Gauge Theory}, math.GT/0402368v9, $(2007)$.

\bibitem{AkSa2} Akbulut, S. and Salur, S., {\it Mirror Duality via $G_2$ and $Spin(7)$ Manifolds}, math.GT/0701790v1, $(2007)$.

\bibitem{ACS} Arikan, M., Cho, H., and Salur, S., {\it Existence of Compatible Contact Structures on $G_2$-Manifolds}, in preparation.

\bibitem{BrGr} Brown, R. and Gray, A., {\it Vector cross products}, Comment. Math. Helv., {\bf $42$}, $(1967)$, pp. $222-236$.

\bibitem{Br1} Bryant, R., {\it Metrics with Exceptional Holonomy}, Annals of Mathematics, Volume $126$, $(1987)$, pp. $526-576$.

\bibitem{Br2} Bryant, R., {\it Some remarks on $G_2$-Structures}, math.DG/0305124v4, $(2005)$.

\bibitem{BrXu} Bryant, R. and Xu, F., {\it Laplacian Flow for Closed $G_2$-Structures: Short Time Behavior}, math.DG/1101.2004v1, $(2011)$.

\bibitem{BrSa} Bryant, R. and Salamon, D., {\it On the Construction of Some Complete Metrics with Exceptional Holonomy}, Duke Math. J., Volume $58$, Number $3$, $(1989)$, $pp. 829-850$.

\bibitem{CMS} Cabrera, F., Monar, M. and Swann, A., {\it Classification of $G_2$-structures}, J. London Math. Soc, {\bf $53$}, $(1996)$, pp. $407-416$.

\bibitem{ClIv} Cleyton, R. and Ivanov, S., {\it On the Geometry of Closed $G_2$-Structures}, math.DG/0306362v3, $(2003)$.

\bibitem{Fe1} Fernandez, M., {\it An example of a compact calibrated manifold associated with the exceptional Lie group $G_2$}, J. Differential Geom., {\bf $26$}, $(1987)$, no. $2$, pp. $367-370$.

\bibitem{Fe2} Fernandez, M., {\it A family of compact solvable $G_2$-calibrated manifolds}, Tohoku Math. J., {\bf $(2)$ $39$}, $(1987)$, no. $2$, pp. $287-289$.

\bibitem{FeGr} Fernandez, M. and Gray, A., {\it Riemannian manifolds with structure group $G_2$}, Ann. Mat. Pura Appl., {\bf $(4)$ $132$}, $(1982)$, pp. $19-45$.

\bibitem{FeIg} Fernandez, M. and Iglesias, T., {\it New examples of Riemannian manifolds with structure group $G_2$}, Rend. Circ. Mat. Palermo, {\bf $(2)$ $35$}, $(1986)$, no. $2$, pp. $276-290$.

\bibitem{Gray} Gray, A., {\it Vector cross products on manifolds}, Trans. Amer. Math. Soc., {\bf $141$}, $(1969)$, pp. $465-504$.

\bibitem{HaLa} Harvey, F.R. and Lawson, H.B., {\it Calibrated Geometries}, Acta. Math. {\bf 148} (1982), 47--157.

\bibitem{Jo} Joyce, D., {\it Compact Manifolds with Special Holonomy}, Oxford Mathematical Monographs, OUP, $(2000)$.

\bibitem{LaMi} Lawson, B. and Michelson, M.-L., {\it Spin Geometry}, Princeton University Press, $(1989)$.

\bibitem{McSa} McDuff, D. and Salamon, D., {\it Introduction to Symplectic Topology}, Oxford University Press, $(1998)$.

\bibitem{Mi} Milnor, J., {\it Spin Structures on Manifolds},  L'Enseignement Mathématique, Vol. $9$, $(1963)$, pp. $198-203$.

\bibitem{daSi} da Silva, A., {\it Lectures on Symplectic Geometry}, Lecture Notes in Mathematics, Springer, $(2001)$.

\end{thebibliography}
\end{document}